\documentclass[11pt]{amsart}
\usepackage{amsfonts,latexsym,amsthm,amssymb,graphicx}
\usepackage[all]{xy}
\usepackage[usenames]{color}
\usepackage{stackengine}
\usepackage{tikz}

\usepackage{hyperref}

\theoremstyle{plain}

\newtheorem{Thm}{Theorem}[section]
\newtheorem{Pro}{Proposition}[section]
\newtheorem{Cor}{Corollary}[section]
\newtheorem{Rem}{Remark}[section]

\newtheorem{Lem}[Thm]{Lemma} \theoremstyle{definition}
\newtheorem{Dfn}{Definition}

\newcommand{\ccsm}[2]{{}}

\newcommand{\crich}[2]{{}}
\newcommand{\lrb}[2]{{\langle #1 , #2 \rangle}}

\newcommand{\nc}{\newcommand}
\nc{\rnc}{\renewcommand}
\nc{\bb}[1]{{\mathbb #1}}
\nc{\bbA}{\bb{A}}\nc{\bbB}{\bb{B}}\nc{\bbC}{\bb{C}}\nc{\bbD}{\bb{D}}
\nc{\bbE}{\bb{E}}\nc{\bbF}{\bb{F}}\nc{\bbG}{\bb{G}}\nc{\bbH}{\bb{H}}
\nc{\bbI}{\bb{I}}\nc{\bbJ}{\bb{J}}\nc{\bbK}{\bb{K}}\nc{\bbL}{\bb{L}}
\nc{\bbM}{\bb{M}}\nc{\bbN}{\bb{N}}\nc{\bbO}{\bb{O}}\nc{\bbP}{\bb{P}}
\nc{\bbQ}{\bb{Q}}\nc{\bbR}{\bb{R}}\nc{\bbS}{\bb{S}}\nc{\bbT}{\bb{T}}
\nc{\bbU}{\bb{U}}\nc{\bbV}{\bb{V}}\nc{\bbW}{\bb{W}}\nc{\bbX}{\bb{X}}
\nc{\bbY}{\bb{Y}}\nc{\bbZ}{\bb{Z}}
\nc{\mbf}[1]{{\mathbf #1}}
\nc{\bfA}{\mbf{A}}\nc{\bfB}{\mbf{B}}\nc{\bfC}{\mbf{C}}\nc{\bfD}{\mbf{D}}
\nc{\bfE}{\mbf{E}}\nc{\bfF}{\mbf{F}}\nc{\bfG}{\mbf{G}}\nc{\bfH}{\mbf{H}}
\nc{\bfI}{\mbf{I}}\nc{\bfJ}{\mbf{J}}\nc{\bfK}{\mbf{K}}\nc{\bfL}{\mbf{L}}
\nc{\bfM}{\mbf{M}}\nc{\bfN}{\mbf{N}}\nc{\bfO}{\mbf{O}}\nc{\bfP}{\mbf{P}}
\nc{\bfQ}{\mbf{Q}}\nc{\bfR}{\mbf{R}}\nc{\bfS}{\mbf{S}}\nc{\bfT}{\mbf{T}}
\nc{\bfU}{\mbf{U}}\nc{\bfV}{\mbf{V}}\nc{\bfW}{\mbf{W}}\nc{\bfX}{\mbf{X}}
\nc{\bfY}{\mbf{Y}}\nc{\bfZ}{\mbf{Z}}
\nc{\bfa}{\mbf{a}}\nc{\bfb}{\mbf{b}}\nc{\bfc}{\mbf{c}}\nc{\bfd}{\mbf{d}}
\nc{\bfe}{\mbf{e}}\nc{\bff}{\mbf{f}}\nc{\bfg}{\mbf{g}}\nc{\bfh}{\mbf{h}}
\nc{\bfi}{\mbf{i}}\nc{\bfj}{\mbf{j}}\nc{\bfk}{\mbf{k}}\nc{\bfl}{\mbf{l}}
\nc{\bfm}{\mbf{m}}\nc{\bfn}{\mbf{n}}\nc{\bfo}{\mbf{o}}\nc{\bfp}{\mbf{p}}
\nc{\bfq}{\mbf{q}}\nc{\bfr}{\mbf{r}}\nc{\bfs}{\mbf{s}}\nc{\bft}{\mbf{t}}
\nc{\bfu}{\mbf{u}}\nc{\bfv}{\mbf{v}}\nc{\bfw}{\mbf{w}}\nc{\bfx}{\mbf{x}}
\nc{\bfy}{\mbf{y}}\nc{\bfz}{\mbf{z}}

\nc{\mcal}[1]{{\mathcal #1}}
\nc{\calA}{\mcal{A}}\nc{\calB}{\mcal{B}}\nc{\calC}{\mcal{C}}\nc{\calD}{\mcal{D}}
\nc{\calE}{\mcal{E}} \nc{\calF}{\mcal{F}}\nc{\calG}{\mcal{G}}\nc{\calH}{\mcal{H}}
\nc{\calI}{\mcal{I}}\nc{\calJ}{\mcal{J}}\nc{\calK}{\mcal{K}}\nc{\calL}{\mcal{L}}
\nc{\calM}{\mcal{M}}\nc{\calN}{\mcal{N}}\nc{\calO}{\mcal{O}}\nc{\calP}{\mcal{P}}
\nc{\calQ}{\mcal{Q}}\nc{\calR}{\mcal{R}}\nc{\calS}{\mcal{S}}\nc{\calT}{\mcal{T}}
\nc{\calU}{\mcal{U}}\nc{\calV}{\mcal{V}}\nc{\calW}{\mcal{W}}\nc{\calX}{\mcal{X}}
\nc{\calY}{\mcal{Y}}\nc{\calZ}{\mcal{Z}}
\nc{\fA}{\frak{A}}\nc{\fB}{\frak{B}}\nc{\fC}{\frak{C}} \nc{\fD}{\frak{D}}
\nc{\fE}{\frak{E}}\nc{\fF}{\frak{F}}\nc{\fG}{\frak{G}}\nc{\fH}{\frak{H}}
\nc{\fI}{\frak{I}}\nc{\fJ}{\frak{J}}\nc{\fK}{\frak{K}}\nc{\fL}{\frak{L}}
\nc{\fM}{\frak{M}}\nc{\fN}{\frak{N}}\nc{\fO}{\frak{O}}\nc{\fP}{\frak{P}}
\nc{\fQ}{\frak{Q}}\nc{\fR}{\frak{R}}\nc{\fS}{\frak{S}}\nc{\fT}{\frak{T}}
\nc{\fU}{\frak{U}}\nc{\fV}{\frak{V}}\nc{\fW}{\frak{W}}\nc{\fX}{\frak{X}}
\nc{\fY}{\frak{Y}}\nc{\fZ}{\frak{Z}}
\nc{\fa}{\frak{a}}\nc{\fb}{\frak{b}}\nc{\fc}{\frak{c}} \nc{\fd}{\frak{d}}
\nc{\fe}{\frak{e}}\nc{\fFf}{\frak{f}}\nc{\fg}{\frak{g}}\nc{\fh}{\frak{h}}
\nc{\fri}{\frak{i}}\nc{\fj}{\frak{j}}\nc{\fk}{\frak{k}}\nc{\fl}{\frak{l}}
\nc{\fm}{\frak{m}}\nc{\fn}{\frak{n}}\nc{\fo}{\frak{o}}\nc{\fp}{\frak{p}}
\nc{\fq}{\frak{q}}\nc{\fr}{\frak{r}}\nc{\fs}{\frak{s}}\nc{\ft}{\frak{t}}
\nc{\fu}{\frak{u}}\nc{\fv}{\frak{v}}\nc{\fw}{\frak{w}}\nc{\fx}{\frak{x}}
\nc{\fy}{\frak{y}}\nc{\fz}{\frak{z}}

\newcommand{\bR}{{\mathbb R}}

\newcommand{\cLL}{\mathcal{L}}

\DeclareMathOperator{\Fun}{Fun}

\begin{document}

\title[Hook formula for Coxeter groups]{Hook formula for Coxeter groups via
the twisted group ring}

\author{Leonardo C.~Mihalcea}
\address{
Department of Mathematics, 
Virginia Tech University, 
Blacksburg, VA 24061
USA
}
\email{lmihalce@vt.edu}

\author{Hiroshi Naruse}
\address{Graduate School of Education, University of Yamanashi, 
Kofu, 400-8510, Japan}
\email{hnaruse@yamanashi.ac.jp}

\author{Changjian Su}
\address{Yau Mathematical Sciences Center, Tsinghua University, Beijing, China}
\email{changjiansu@mail.tsinghua.edu.cn}

\maketitle

\begin{abstract}
We use Kostant and Kumar's twisted group ring and its dual to formulate 
and prove a generalization of Nakada's colored hook formula for any Coxeter groups. For dominant minuscule elements of the Weyl group of a Kac--Moody algebra, this provides another short proof of Nakada's colored hook formula.
\end{abstract}

\section{Introduction}

The purpose of this note is to give a short algebraic proof of Nakada's colored hook formula 
(Corollary \ref{col:Nak})
and its generalization
to any Coxeter groups (Theorem \ref{thm:hookJ}).
As a by-product of our formulation, we get a simple proof of
Shi's Yang--Baxter relations  \cite{Shi} in the group algebra
$Q(V)[W]$ for a Coxeter group $W$ (see Remark \ref{Shi_simplified} below), where $V$ is the underlying vector space of the root datum of $W$, see Section \ref{sec:Cexeter}.

For the proof, we use Kostant--Kumar's twisted group ring $H_Q$ (see \ref{H_Q}), 
$\cLL_w\in H_Q$ (cf. Definition\ref{def:L_w}), the dual basis $\eta^w$ (cf. Equation \ref{duality}), 
and a Molev--Sagan recursion formula (\ref{eq:recLR})
(cf. \cite[Prop. 3.2]{MS99}) .
The $\cLL_w$ and $\eta^w$
 are algebraic counterparts of some
geometric objects studied in \cite{MNS22}.
In Appendix \ref{sec:app}, we give a brief explanation of the geometric background of this construction in the finite Weyl group case.

\section{Coxeter group and root system}\label{Cox:root}

In this section, we recall some fundamental properties of the Coxeter group and
its root system. 

\subsection{Coxeter group and root system}\label{sec:Cexeter}

Let $(W,S)$ be a Coxeter system, where $S=\{s_i\}_{i\in I}$ is the set of generators. For any $w\in W$, the support of $w$ (the set of generators in $S$ which appear in some reduced expression of $w$) is a finite set. Hence, for the purpose of generalizing the hook formula (Theorem \ref{thm:hookJ}), we can assume $I:=\{1,2,\ldots, r\}$ is a finite set.
Let $(V,\Sigma,\Sigma^\vee)$ be a triple (called the root datum of $W$)
with the following properties:
\begin{itemize}
\item[($R_0$)] $V$ is a finite-dimensional vector space over $\mathbb R$ which is
 a representation space of $W$.
\noindent
Let {$V^{*}=\nobreak {\rm Hom}_{\mathbb R} (V, \mathbb R)$}, with
 the natural pairing 
 
 $( \cdot,\cdot ) : V\times V^{*}\to \mathbb R$.
\item[($R_1$)] $\Sigma=\{\alpha_1,\ldots,\alpha_r\}\subset V$ and
$\Sigma^\vee=\{\alpha_1^\vee,\ldots,\alpha_r^\vee \}\subset 
V^{*}$. The elements in $\Sigma$ and
$\Sigma^\vee$ are assumed to be linearly independent.

\item[($R_2$)] $( \alpha_i,\alpha_i^\vee)=2$ for $i=1,2,\ldots, r$.

\item[($R_3$)]  $W$ acts on $V$  by 
$s_i \lambda=\lambda-( \lambda,\alpha_i^\vee) \alpha_i$ 
for $\lambda\in V$ and  $i=1,2,\ldots, r$.

\end{itemize}

\noindent
Note that such a triple $(V,\Sigma,\Sigma^\vee)$ always exists.
For example, for a Coxeter system $(W,S)$,
 a ``Cartan'' matrix $C=(c_{i,j})_{r\times r}$ whose entry is presumed to be  
$c_{i,j}=(\alpha_j,\alpha_i^\vee)$, can defined as follows.
Let $M=(m_{i,j})_{r\times r}$ be the Coxeter matrix of $(W,S)$
whose entry $m_{i,j}$ is the order of $s_i s_j$.
By definition, $m_{i,j}=m_{j,i}\in \{1,2,3,,\ldots\}\cup\{\infty\}$ and  $m_{i,i}=1$ for 
$1\leq i,j\leq r$.
We define a ``Cartan'' matrix $C=(c_{i,j})_{r\times r}$ whose entry $c_{i,j}$ is a real number satisfying the
following conditions.

\begin{itemize}
\item[(1)] $c_{i,i}=2$ ($1\leq i \leq r$).

\item[(2)] if $i\neq j$ then  $c_{i,j}\leq 0$, 
and $c_{i,j}=0$ if $m_{i,j}=2$. 

\item[(3)] for $m_{i,j}>2$, it is required that \\
 $c_{i,j} c_{j,i}=4\cos\left(\displaystyle\frac{\pi}{m_{i,j}} \right)^2$.
\end{itemize}

\noindent
For a crystallographic Coxeter group $W$ 
($m_{i,j}\in \{1, 2,3,4,6,\infty\}  \text{ for all }1\leq  i,j \leq r$), 
such as the Weyl group of a Kac--Moody algebra,
we can take all $c_{i,j}$ to be integers.
For a standard choice of an arbitrary Coxeter group, we can take $C$ to be symmetric ($c_{i,j}=c_{j,i}=
-2\cos(\pi/m_{i,j})$). Then using the argument of \cite[Proposition 1.1]{Kac} over the real numbers,
 we have such a
triple $(V,\Sigma,\Sigma^\vee)$ for $(W,S)$.

\vspace{0.1cm}

Let $R=\{w(\alpha_i)\mid w\in W, i=1,\ldots,r\}\subset V$ be the set roots of $(W,S)$.
It is known that we have a disjoint union
$R=R^{+} \sqcup R^{-}$, where 
$R^{+}=\left\{\alpha\in R \mid
\alpha=
\sum_{i=1}^{r} c_i \alpha_i , c_i\geq 0\right\}$ 
is the set of positive roots and $R^{-}=-R^{+}$
(cf. \cite{Deo} for basic properties of the root system of a Coxeter group.).
The action of $W$ on $V^{*}$ is defined by
$$s_i(y)=y-( \alpha_i, y ) \alpha_i^\vee \text{ for } y\in V^{*} 
\text{ for } i=1,2,\ldots ,r.$$
For a root $\gamma \in R$, the dual root $\gamma^\vee\in V^{*}$ is defined by $\gamma^\vee=w(\alpha_i^\vee)$
if $\gamma=w(\alpha_i)$.
For a positive root $\beta\in R^{+}$,  let $s_\beta \in W$ be the
corresponding reflection, i.e.
$$s_\beta(x)=x-(x, \beta^\vee) \beta \text{ for } x\in V,$$
$$s_\beta(y)=y-( \beta, y ) \beta^\vee \text{ for } y\in V^{*}.$$
\noindent
The pairing $( \cdot,\cdot )$ is $W$-invariant, i.e.
$$( w (x), w(y) )=( x,y ) \text{ for all } x\in V, y\in V^{*}, w\in W.$$

\subsection{Some lemmas} 
In this section, we list some results about the Coxeter group and the root system. For any $w\in W$, let $S(w):=R^{+} \cap w R^{-}$.
\begin{Lem}\label{beta}
For  a reduced expression $w=s_{i_1} s_{i_2}\cdots s_{i_\ell}$,
let 
\begin{equation}\label{beta_j}
\beta_j:=s_{i_1}\cdots s_{i_{j-1}} \alpha_{i_j} \;\; (1\leq j\leq \ell).
\end{equation}
Then
$ |S(w)|=\ell \text{ and }
S(w)=\{ \beta_1,\beta_2,\ldots, \beta_{\ell}\}.$
\end{Lem}
\begin{proof}
This is a consequence of \cite[Theorem 5.4]{Hum},
i.e., for $u\in W$, $\ell(u s_i)>\ell(u)\iff u(\alpha_i)>0$.
\end{proof}

\begin{Rem}\label{rem:faithful}
By this Lemma, we see that the representation of $W$ on $V$ is faithful.
\end{Rem}

 Let $J\subset S$ be a subset of generators of $W$ and $W_J$ be the subgroup generated by 
 $\{s_i\}_{i\in J}$.
Let $W^J:=W/W_J$ be the set of minimal length coset representatives.
Let us denote by $<$ the Bruhat order on $W$ (\cite[5.9]{Hum}). 
Then it induces a
Bruhat order on $W^J$.
An element $\chi\in V$ is said to be dominant if
$(\chi, \alpha_j^\vee)\geq 0$ for any $1\leq j \leq r$.

\begin{Lem}\label{lem:non-zero}
\cite[Ch.5, \S 4.6]{Bur}
If $\chi\in V$ is dominant,
the stabilizer subgroup 
$stab_{W}(\chi) :=\{w\in W\mid w(\chi)=\chi\}$ is equal to $W_J$ for some 
$J\subset S$, 
i.e.,  for $x\neq y\in W^J$, 
$x(\chi)-y(\chi)\neq 0$.
\end{Lem}
\begin{Lem}\label{lem:Bruhat_order}
\cite[Lemma 4.1]{FW04}
Asume $\chi\in V$ is dominant with $stab_{W}(\chi)=W_J$.
For $x<y\in W^J$, the two conditions below are equivalent.

(a) $\exists \gamma\in R^{+}$ such that
$y W_J= x s_\gamma W_J$,
 
(b) $\exists \beta\in R^{+}$ such that
$y W_J= s_\beta x W_J$.

\noindent
Moreover, if these conditions are satisfied,
then 
$$x(\chi)-y(\chi)=(\chi, \gamma^\vee) \beta.$$
In particular, $\beta$ is unique if it exists.
\end{Lem}

\section{Twisted group ring and its dual}

\subsection{Kostant--Kumar twisted group ring $H_Q$}\label{H_Q}
Assume $(W,S)$ is a Coxeter system with root datum $(V,\Sigma,\Sigma^\vee)$ (cf. Section \ref{Cox:root}).
The Coxeter group analog of the Kostant--Kumar's twisted group algebra is
$H_Q
:=Q(V)\rtimes {\mathbb R}[W]$, the smash product of the fraction field $Q(V)$ of the symmetric algebra $S(V)$ with the group algebra $ {\mathbb R}[W]$.
As a vector space over $\bR$,
$H_Q=Q(V)\otimes_{\mathbb R} {\mathbb R}[W]$, 
with $Q(V)$-free basis $\{\delta_w \}_{w\in W}$, 
and the multiplication is defined
as follows. For $a=\sum_{w\in W} a_w \delta_w , b=\sum_{u\in W} b_u \delta_u \in H_Q$,
$$a\cdot  b=\sum_{w,u\in W} a_w w(b_u) \delta_{w u}.
$$
If $W$ is the Weyl group of a Kac--Moody algebra with an integral Cartan matrix,
$H_Q\otimes \mathbb C$ is the twisted group ring introduced by Kostant--Kumar \cite{KK86},
where the divided difference operator 
$\partial_i=\frac{1}{\alpha_i} \delta_{s_i}-\frac{1}{\alpha_i}\delta_{id}$ 
and the associated  basis element $\partial_w \in H_Q$ ($w\in W$) are defined.
Here we introduce another (inhomogeneous) basis $\calL_w$,
which is the main tool of our calculation.

\begin{Dfn}\label{def:L_w}
For $i=1,2,\ldots r$,
define $\cLL_i\in H_{Q}$  by $\cLL_i:=\frac{1+\alpha_i}{\alpha_i} \delta_{
s_i}-\frac{1}{\alpha_i}\delta_{id}$, i.e.
$\cLL_i=\partial_i+\delta_{s_i}$.
\end{Dfn}
From the definition,
$\delta_{id}+\alpha_i \cLL_i=(1+\alpha_i) \delta_{s_i}$ and 
$\cLL_i^2=\delta_{id}$. 

\begin{Pro}\label{prop:DL}
 For $w\in W$, the following holds.

\begin{itemize}
\item[(a)$_w$] 
Let $w=s_{i_1}\cdots s_{i_\ell}\in W$ be a reduced expression, and define
$\cLL_w:=\cLL_{i_1}\cdots \cLL_{i_\ell}$. Then
it does not depend on the choice of the reduced expression of $w$.

\item[(b)$_w$] 
For $\chi\in V\subset S(V)$, the following equality holds. (We abbreviate $\chi \delta_{id}$ as  $\chi$.)
\begin{equation}
\cLL_w \chi=w(\chi) \cLL_{w}-
\displaystyle\sum_{\gamma\in S(w^{-1})}( \chi,\gamma^\vee) 
\cLL_{w s_\gamma}.
\end{equation}

\item[(c)$_w$]
Let $\cLL_w=\displaystyle\sum_{v\in W} e_{w,v}  \delta_{v},\; e_{w,v}\in Q(V)$. 
Then
\begin{center}
$e_{w,v}=0$ unless $v\leq w$, and
$e_{w,w}=\displaystyle\prod_{\beta\in S(w)} \frac{1+\beta}{\beta}$.
\end{center}
\end{itemize}

\end{Pro}

\begin{proof}

We will prove (a)$_w$,(b)$_w$,(c)$_w$ simultaneously,
by induction on length $\ell(w)$ of $w$ as in
\cite{KK86}.
If $\ell(w)=1$, i.e., $w=s$ for $s\in S$, we can check the formulae  (a)$_{s}$,(b)$_{s}$,(c)$_{s}$ directly.
Assume $\ell(w)>1$ and $u=s w<w$, $s\in S$.
Then we have, by (b)$_u$ and (b)$_{s}$,

$$\cLL_s \cLL_{u} \chi=
w(\chi) \cLL_s \cLL_{u}
-(u(\chi),\alpha_s^\vee)\cLL_u
-
\displaystyle\sum_{\gamma\in S(u^{-1})}( \chi,\gamma^\vee) 
\cLL_s \cLL_{u s_\gamma}. $$

As $S(w^{-1})=S(u^{-1})\cup \{u^{-1}\alpha_s\}$, and (a)$_z$ for elements $z$ of length less than $\ell(w)$,
we have 
$$\cLL_s \cLL_{u} \chi-w(\chi) \cLL_s \cLL_{u}=
\displaystyle -\sum_{\gamma\in S(w^{-1})} ( \chi,\gamma^\vee) \cLL_{w s_\gamma} ,$$
which gives part (b)$_w$.
Likewise, for $s'\in S$ with $v=s' w <w$, we have
$$\cLL_{s'} \cLL_{v} \chi-w(\chi) \cLL_{s'} \cLL_{v}=
\displaystyle -\sum_{\gamma\in S(w^{-1})} ( \chi,\gamma^\vee) \cLL_{w s_\gamma} .$$
Therefore,
\begin{equation}\label{eq:su-s'v}
\cLL_s \cLL_{u} \chi-w(\chi) \cLL_s \cLL_{u}= 
\cLL_{s'} \cLL_{v} \chi-w(\chi) \cLL_{s'} \cLL_{v}.
\end{equation}
\noindent
Write $\cLL_s \cLL_{u} = \sum_{x\in W} q_x \delta_x$ , and
$\cLL_{s'} \cLL_{v} = \sum_{x\in W} q'_x \delta_x$ $(q_x,q'_x\in Q(V))$.
Then by (c)$_u$ and (c)$_v$, we have
$q_x=q'_x=0$ unless $x\leq w$ and

\begin{equation}\label{eq_w}
q_w=\prod_{\beta\in S(w)}\frac{1+\beta}{\beta}=q'_w,
\end{equation}
which gives part (c)$_w$.
From (\ref{eq:su-s'v}) we have
$$(x(\chi)- w(\chi))q_x=( x(\chi) - w(\chi))q'_x \text{ for } \forall x \in W.$$
\noindent
As $V$ is faithful (Remark \ref{rem:faithful}), we have $q_x=q'_x$ for $x\neq w$.
Together with (\ref{eq_w}), we have
$$\cLL_s \cLL_{u} = \cLL_{s'} \cLL_{v} ,$$
which proves part  (a)$_w$.
\end{proof}

By this Proposition, $\{\cLL_w\}_{w\in W}$ forms a 
basis of the left $Q(V)$-module $H_Q$, and if we expand
$\cLL_w \chi$ in this basis
\begin{equation}\label{chev}
\cLL_w \chi=\displaystyle\sum_{v\in W} c^w_{\chi,v} \cLL_v, \;\;c^w_{\chi,v}\in Q(V), \chi\in V,
\end{equation}
we have
\begin{equation}\label{coef:c}
\begin{array}{ccc}
c_{\chi,v}^{w}=
\left\{\begin{array}{ll}
w(\chi)& \text{ if } v=w\\
-(\chi, \gamma^\vee) &\text{ if } v<w=v s_\gamma, \gamma\in R^{+}\\
0& \text{ otherwise }\\
\end{array}.
\right.
\end{array}
\end{equation}

\subsection{Dual basis $\{\eta^w\}_{w\in W}$ of $\{\cLL_w\}_{w\in W}$}\label{sec:dual}
Let $\Fun(W, Q(V))$ denote the ring of functions on $W$ with values in $Q(V)$, with natural $Q(V)$-module structure by
$ (q \xi)(w)= q \xi(w)$ for $q\in Q(V), w\in W$, and $\xi\in  \Fun(W, Q(V))$.
There is a perfect pairing
$\langle \cdot  , \cdot \rangle: H_Q\times
\Fun(W, Q(V))
\to Q(V)$, given by
\begin{equation}\label{def_pairing}
{\lrb{a}{\xi}=\displaystyle\sum_{w\in W} a_w \xi(w),}
\hspace{3cm}\\
\end{equation}
$\text{ for }a=\displaystyle\sum_{w\in W} a_w \delta_w\in H_Q
\text{ and }\xi\in  \Fun(W, Q(V)).$
Here a pairing 
$\langle ,\rangle : M_1\times M_2\to Q(V)$ is perfect 
if it induces an isomorphism $M_1^*\simeq M_2$ of $Q(V)$-modules.
\begin{Lem}\label{adjoint}
For $\chi\in V$,
define  $L_\chi\in \Fun(W,Q(V))$ by $L_\chi(w):=w(\chi)$. Then we have
\begin{equation}
\langle h  \chi , f \rangle=\langle h, L_{\chi} f\rangle \hspace{0.5cm} 
\end{equation}
$\text{ for } h\in H_Q, f \in \Fun(W,Q(V)).$
\end{Lem}
\begin{proof}
If $h=\displaystyle\sum_{w\in W} a_w \delta_w$, then $h \chi=\displaystyle\sum_{w\in W} a_w w(\chi) \delta_w$.
Therefore
we have
$\langle h \chi , f \rangle=\displaystyle\sum_{w\in W} a_w w(\chi) f(w)=
\sum_{w\in W} a_w (L_\chi f)(w)=\langle h  , L_\chi f \rangle.$
\end{proof}

We can define element $\eta^v\in \Fun(W, Q(V))$ for $v\in W$  by duality
\begin{equation}\label{duality}
\lrb{\cLL_w}{\eta^{v}}=\delta_{w,v}
\hspace{1cm} \text{ for } \forall w\in W.
\end{equation}
Using this duality and Proposition \ref{prop:DL} (c)$_w$, we have

\begin{equation}\label{eta}
\eta^v(w)=0 \text{ unless } v\leq w \text{, and } \eta^w(w)=\prod_{\beta\in S(w)} \frac{\beta}{1+\beta}.
\end{equation}

Let $J\subset S$ be a subset of the generators of $W$.

\begin{Dfn}\label{para_J}
 For $v\in W^J$, define $\eta^v_J\in \Fun(W^J, Q(V))$ by
\begin{equation}\label{def_J}
\eta^v_J(w):=\sum_{u\in v W_J, u\leq w} \eta^u(w) \text{ for } w\in W^J.
\end{equation}
\end{Dfn}
We can formally write $\eta^v_J=\displaystyle\sum_{u\in v W_J} \eta^u$, as  $ \eta^u(w)=0$ if 
$u\not\leq w$.
From (\ref{eta})
it follows that for $v,w\in W^J$,
\begin{equation}\label{eta_J}
\eta^v_J(w)=0 \text{ unless } v\leq w \text{, and } \eta^w_J(w)=\prod_{\beta\in S(w)} \frac{\beta}{1+\beta}.
\end{equation}
Thus, $\{\eta_J^w\mid w\in W^J\}$ is a basis for $\Fun(W^J, Q(V))$ over $Q(V$).

For any $v,z\in W^J$ and $\chi\in V^{W_J}$,
define 	\begin{equation}\label{def_c^J}
		c_{\chi,v}^{z,J}:=\sum_{u\in W_J} c_{\chi,v}^{zu}.
\end{equation}
By the equalities $(\ref{coef:c})$ and $(\ref{def_c^J})$,
we have the following equality.
\begin{Lem}\label{lem:coeffJ}
		\[c_{\chi,v}^{z,J}=\begin{cases} v\chi, & \textit{ if } v=z;\\
			-(\chi,\gamma^\vee), & \textit{ if } vW_J<vs_\gamma W_J=zW_J,\gamma\in R^+;\\
			0, & \textit{otherwise.}
			\end{cases}
	 \]
\end{Lem}

\begin{Pro}\label{chev:eta}
	For $\chi\in V^{W_J}$, define $L^J_{\chi}\in \Fun(W^J,Q(V))$ by $L^J_{\chi}(z)=z(\chi)$ for $z\in W^J$.
	Then for any $v\in W^J$,
	\begin{equation}\label{chev_J}
		(L^J_{\chi} \eta^v_J) (w)=\displaystyle\sum_{z\in W^J, v\leq z\leq w} c_{\chi,v}^{z,J} \eta^z_J (w),
		\text{ for }w\in W^J.
	\end{equation}
	We can formally write $L^J_{\chi} \eta^v_J=\displaystyle\sum_{z\in W^J, v\leq z} c_{\chi,v}^{z,J} \eta^z_J$.
\end{Pro}

\begin{proof}
	Combining duality (\ref{duality}), Equation (\ref{chev}) and Lemma \ref{adjoint}, we get
	\begin{equation} (L_{\chi} \eta^v) (w)=\displaystyle\sum_{z\in W, v\leq z\leq w} c_{\chi,v}^{z} \eta^z (w),
		\text{ for }w\in W.
	\end{equation}
    Moreover, it is easy to see that $c_{\chi,vu}^{zy}=c_{\chi,v}^{zyu^{-1}}$ for any $z,v\in W^J$ and $y,u\in W_J$. Then Equation (\ref{chev_J}) readily follows by these observations.
\end{proof}

For $u,v,w\in W^J$, let $d^{w,J}_{u,v}\in Q(V)$ be the structure constants, i.e.,
\begin{equation}\label{lr_d}
\eta^u_J \eta^v_J=\sum_{w\in W^J} d^{w,J}_{u,v} \eta^w_J.
\end{equation}

\begin{Lem}\label{d^{w,J}_{u,v}}
 The coefficients $d^{w,J}_{u,v}\in Q(V)$ have the following properties.

\begin{itemize}
\item[(i)] For $u,v,w\in W^J$, $d^{w,J}_{u,v}=0$ unless $u\leq w$ and $v\leq w$.\\
\item[(ii)] For $v,w\in W^J$, $d^{w,J}_{v,w}=\eta^v_J(w)$,\\
\item[(iii)] For $w\in W^J$, $d^{w,J}_{w,w}=\displaystyle\prod_{\beta\in S(w)} \frac{\beta}{1+\beta}$.
\end{itemize}
\end{Lem}
\begin{proof}
Given $u,v\in W^J$, if there is a $w\in W^J$ such that $u\not \leq w$ or
$v\not\leq w$ and $d^{w,J}_{u,v}\neq 0$, take a minimal such $w$ in Bruhat order and evaluate both sides of  (\ref{lr_d}) at $w$.
The left hand side becomes 0, but right hand side
is nonzero by (\ref{eta_J}), which gives a contradiction. Therefore (i) holds.
As $\eta^v_J(w) \eta^w_J (w)=d_{v,w}^{w,J} \eta^{w}_J(w)$,  and $\eta^{w}_J(w)=\eta^{w}(w)\neq 0$, we have (ii).
Finally, (iii) follows from Equation (\ref{eta_J}).

\end{proof}

\begin{Pro}
\label{lem:recLR}
For any 
$u,v,w\in W^J$ and any vector $\chi\in V^{W_J}$, 
the following holds:

$\displaystyle
(c^{w,J}_{\chi,w} - c^{u,J}_{\chi,u}) d^{w,J}_{u,v}=
\sum_{\substack{u< x\leq w},x\in W^J} 
c^{x,J}_{\chi,u}d^{w,J}_{x,v}-\sum_{\substack{ u,v\leq y<w,y\in W^J}} c^{w,J}_{\chi,y} d^{y,J}_{u,v} \/.
$
\end{Pro}

\begin{proof} By taking the coefficient of $\eta^w_J$ in
$L_\chi (\eta^u_J \eta^v_J) =(L_\chi \eta^u_J) \eta^v_J$,
we have

$\displaystyle c^{w,J}_{\chi,w} d^{w,J}_{u,v}+\sum_{u ,v\le y < w,y\in W^J} c^{w,J}_{\chi,y} d^{y,J}_{u,v}
=c^{u,J}_{\chi,u} d^{w,J}_{u,v}+\sum_{u < x \le w,x\in W^J} c^{x,J}_{\chi,u} d^{w,J}_{x,v}$,

from which the assertion holds.
\end{proof}

\begin{Cor}
\label{lem:recLR}
If $\chi\in V^{W_J}$ 
satisfies $c^{w,J}_{\chi,w} \neq c^{u,J}_{\chi,u}$ 
(e.g. $\chi=\pi_J$, cf.  Lemma \ref{lem:non-zero}.),
 then we have
$$
d^{w,J}_{u,v}=\displaystyle\frac{1}{ c^{w,J}_{\chi,w} - c^{u,J}_{\chi,u} }\times
\left(\sum_{u < x \le w, x\in W^J} c^{x,J}_{\chi,u}d^{w,J}_{x,v}-\sum_{u,v \le y < w, y\in W^J} c^{w,J}_{\chi,y} d^{y,J}_{u,v}\right).
$$

In particular, for the case $v=w$,
\begin{equation}
\label{eq:recLR}
d^{w,J}_{u,w}=\sum_{x\in W^J, u < x \le w}\frac{c^{x,J}_{\chi,u} }{ c^{w,J}_{\chi,w} - c^{u,J}_{\chi,u} }\; d^{w,J}_{x,w}.
\end{equation}
\end{Cor}

\section{Yang--Baxter elements in the group algebra $Q(V)[W]$}

Let $Q(V)[W]$ denote the group algebra of $W$ over $Q[V]$.
The purpose of this section is to relate $H_Q$ with $Q(V)[W]$ through a left $Q(V)$-module homomorphism
and prove the equations in Corollary \ref{sum=1}.
As a byproduct, we get a simple proof for the Yang--Baxter relations for the Coxeter groups (
\cite[Proposition\ref{prop:Shi}]{Shi}).

Let $\Delta_i:=(1+\alpha_i) \delta_{s_i} =1+\alpha_i \cLL_i\in H_Q$. For a reduced expression $w=s_{i_1}\cdots s_{i_\ell}$, define
\begin{equation}\label{Delta_w}
\Delta_w:=\Delta_{i_1} \Delta_{i_2} \cdots \Delta_{i_\ell}\in H_Q.
\end{equation}

Then it is easy to see that
$\Delta_w=A(w) \delta_w$, where $A(w)=\displaystyle\prod_{j=1}^\ell (1+ \beta_j)$ for $\beta_j$ defined in (\ref{beta_j}).
By Lemma \ref{beta},
$\Delta_w$ does not depend on the choice of the reduced expression for $w$.
We can  expand $\Delta_w$ in terms of $\{\cLL_v\}_{v\in W}$ as follows.
\begin{equation}\label{Delta_to_cLL}
\Delta_w=\sum_{v\in W} q(v,w) \cLL_v,\;\;q(v,w)\in Q(V).
\end{equation}

\begin{Lem}\label{prop_q}
The coefficients $q(v,w)$ satisfy the following properties.

\begin{itemize}
\item[(i)]\label{prop_q_loc} $q(v,w)=A(w) \eta^v(w).$
\item[(ii)]\label{prop_q_support} $q(v,w)=0$ unless $v\leq w$.
\item[(iii)]\label{prop_q_rec} If $s_i w>w$,
\begin{equation}
q(u,s_i w)=s_i(q(u,w))+\alpha_i s_i(q(s_i u,w)).
\end{equation}
\end{itemize}
\end{Lem}
\begin{proof}
(i) follows by  evaluation of $\langle \Delta_w, \eta^v\rangle$, using definitions (\ref{def_pairing}) and  (\ref{duality}). 
(ii) follows from (i) and (\ref{eta}).
 If $s_i w>w$, by 
comparing the coefficient  of $\cLL_u$ in 

$\Delta_{s_i w}=\Delta_{s_i} \Delta_w=(1+\alpha_i )\delta_{s_i} \displaystyle
\sum_{u\in W} q(u,w) \cLL_u\\
=\sum_{u\in W}s_i(q(u,w)) (1+\alpha_i \cLL_i)\cLL_{u},$

\noindent
we get the relation (iii), as $(\cLL_i)^2=\delta_{id}$.

\end{proof}

We can now define a left $Q(V)$-module isomorphism $\Phi: H_Q\to Q(V)[W]$ by
sending
$$\Phi\left(\sum_{w\in W} c_w \cLL_w\right)=\sum_{w\in W} c_w w,\;\; c_w\in Q(V).$$

\begin{Dfn}
For $w\in W$, define a Yang--Baxter element $Y_w:=\Phi(\Delta_w)\in Q(V)[W]$. 
\end{Dfn}
Then by (\ref{Delta_to_cLL}), 
\begin{equation}\label{def_Y}
Y_w=\sum_{v\leq w} q(v,w) v.
\end{equation}

\begin{Pro}\label{prop:Shi}\cite{Shi}
For a reduced expression $w=s_{i_1}s_{i_2}\cdots s_{i_\ell}\in W$,
let $\beta_j$ be as in (\ref{beta_j}).
Then the following equality holds in
$Q(V)[W]$.
\begin{equation}\label{prod_Y}
Y_w=(1+\beta_{1} s_{i_1})(1+\beta_{2} s_{i_2})\cdots (1+\beta_{\ell} s_{i_\ell}).
\end{equation}

\end{Pro}
\begin{proof}
It follows directly by induction on $\ell(w)$ and Lemma \ref{prop_q_rec} (iii).
\end{proof}

\begin{Rem}\label{Shi_simplified}
The proof of the above Proposition gives a simple proof for \cite[Theorem 3.2]{Shi}, by replacing $\alpha_i$ with $-\alpha_i$ for all $i\in I$.
\end{Rem}

\begin{Cor}\label{sum=1}
 For $w\in W$, we have
\begin{equation}\label{A=sum_q}
A(w)=\sum_{v\leq w} q(v,w),
\end{equation}
\begin{equation}\label{sum_eta}
\sum_{v\leq w} \eta^v(w)=1,
\end{equation}
\begin{equation}\label{sum_eta^J}
\text{ if }w\in W^J, 
\sum_{v\in W^J, v\leq w} \eta^v_J(w)=1.
\end{equation}
\end{Cor}
\begin{proof}
Because of the Coxeter relations for the generators $s\in S$, there is a $Q(V)$-algebra homomorphism $ev: Q(V)[W]\to Q(V)$ defined by $ev(s)=1 \text{ for } \forall s\in S$. Hence,
$ev(Y_w)=\displaystyle\sum_{v\leq w} q(v,w)$ by equation (\ref{def_Y}),
and $ev(Y_w)=A(w)$ by equation (\ref{prod_Y}). Therefore we have equality (\ref{A=sum_q}).
Dividing both sides of equality (\ref{A=sum_q}) by $A(w)$, we get
the equality (\ref{sum_eta}) by Lemma \ref{prop_q} (i).
The equality (\ref{sum_eta^J}) follows from (\ref{sum_eta}) and the definition of 
$\eta^v_J(w)$ (\ref{def_J}).
\end{proof}

\section{Main Theorem}
In the same setup as in Section \ref{sec:dual}, for any $x,y\in W^J$, denote by
$x \overset{\beta}{\to} y$ if $yW_J=s_\beta xW_J$ and $x<y$. Then we have the following formula.
\begin{Thm}\label{thm:hookJ}(Hook formula for Coxeter group)
Let $\chi\in V$ be dominant with stabilizer subgroup $stab_{W}(\chi)=W_J$.
For any $w\in W^J$, the following equality holds.
\begin{equation}\label{hook_formula}
\begin{split}
\sum \frac{m_k}{m_1\beta_1+m_2\beta_2+\cdots +m_k\beta_k} \cdot \ldots \cdot 
\frac{m_1}{m_1\beta_1}\\
=\prod_{\beta \in S(w)}\left(1+\frac{1}{\beta}\right),
\end{split}
\end{equation}
where the sum is over all directed paths 
\begin{equation}\label{sequence}
\begin{split}
 x_k \overset{\beta_k}{\to} x_{k-1} \overset{\beta_{k-1}}{\to} \ldots \overset{\beta_1}{\to} x_0 = w\text{ in }W^J, \\
\end{split}
\end{equation} for any integer $k \ge 0$,
and $m_i: =( \chi, \gamma_i^\vee )$ for the unique
$\gamma_i\in R^{+}$ such that $x_{i-1}W_J=x_i s_{\gamma_i} W_J$\;\;$(1\leq i\leq k)$.

Taking the lowest degree terms in equation (\ref{hook_formula}), we get
\begin{equation}
\begin{split}
\sum \frac{m_k}{m_1\beta_1+m_2\beta_2+\cdots +m_k\beta_k} \cdot \ldots \cdot 
\frac{m_1}{m_1\beta_1}\\
=\prod_{\beta \in S(w)}\frac{1}{\beta},
\end{split}
\end{equation}
where the sum is over all directed sequences as in (\ref{sequence}) with
length $k=\ell(w)$.

\end{Thm}

\begin{proof} 
By Lemma \ref{lem:non-zero}, $c^{w,J}_{\chi,w} - c^{u,J}_{\chi,u} =w\chi-u\chi\neq 0$ for any $u<w\in W^J$.
Therefore we can apply the formula (\ref{eq:recLR}) recursively to get

\begin{equation}\label{sum_d}
d^{w,J}_{u,w}=\\
\sum
\left(\prod_{i=1}^{k}
\frac{c^{x_i,J}_{\chi,x_{i-1}} }{ c^{w,J}_{\chi,w} - c^{x_{i-1},J}_{\chi,x_{i-1}} }\; \right)
d^{w,J}_{w,w},
\end{equation}
where the summation is over all integers $k \ge 1$, and
sequences
$u=x_0<x_1<x_2<\cdots <x_k=w, x_i\in W^J$.
On the other hand, by (\ref{sum_eta^J}) and Lemma \ref{d^{w,J}_{u,v}},
$\sum_{u\leq w}d^{w,J}_{u,w}=1$ and we have equality
 $\displaystyle\sum_{u\in W^J, u\leq w} \frac{d^{w,J}_{u,w}}{d^{w,J}_{w,w}}
=\prod_{\beta\in S(w)} \frac{1+\beta}{\beta}$.
Then the theorem follows from Lemma \ref{lem:Bruhat_order} and Lemma \ref{lem:coeffJ}.

\end{proof}
An element $w\in W$ is said to be $\chi$-minuscule $(\chi\in V)$ if
$$(s_{i_{j+1}}s_{i_{j+2}}\cdots s_{i_\ell}(\chi), \alpha_{j}^\vee)=1
\;\; (j=1,\ldots \ell)$$
 for a reduced decomposition
$w=s_{i_1}\cdots s_{i_\ell}$.
\begin{Lem} (\cite[Lemma 3.5, Corollary 3.10]{MNS22})
In the setting of Theorem \ref{thm:hookJ},
if $u\in W$ is $\chi$-minuscule, then 
$u\in W^J$ and
$m_i=1$\;\;$(1\leq i\leq k)$
for
each directed path (\ref{sequence}).

\end{Lem}

\begin{Cor}\label{col:Nak}(Nakada's colored hook formula \cite[Theorem 7.1]{Nak})
 Let $W$ be the Weyl group of a Kac--Moody algebra,
with the set of simple reflections $S$, acting on the real Cartan subalgebra $\frak h_{\mathbb R}$.
Let $\lambda \in V=\frak h^{*}_{\mathbb R}$ be a dominant integral weight,
$stab_{W}(\lambda)=W_J$,
and $w\in W$ be a $\lambda$-minuscule element.
Then we have 
\begin{equation}\label{E:minusculeNak}
\sum \frac{1}{\beta_1+\beta_2+\cdots +\beta_k} \cdot \ldots \cdot 
\frac{1}{\beta_1}
=\prod_{\beta \in S(w)}\left(1+\frac{1}{\beta}\right),
\end{equation}
where the sum is over all integers $k \ge 0$ and directed paths (\ref{sequence}).
\end{Cor}
\begin{Rem}
(i) Nakada's original formula is written in terms of pre-dominant integral weights.
The equivalence of the formulation above is explained in \cite[Section  3.2]{MNS22}.

(ii) The terminology `hook'  arises from the  Grassmannian situation, 
where we may consider each $\beta\in S(w)$ as a hook in the partition for $w$.

(iii) There is a $K$-theoretic analog of Theorem. 5.1,  which will be considered elsewhere.
\end{Rem}

\section{Appendix: Geometric interpretation}\label{sec:app}
Here we comment briefly on the geometric interpretation of our construction
for the finite Weyl group case.
In \cite{MNS22}, we noticed that for a finite Weyl group $W$,
Nakada's colored hook formula can be derived using geometric arguments
via the Chern--Schwartz--MacPherson (CSM) classes of the Schubert cells.
 Let $G$ be a reductive algebraic group with Borel $B$, maximal torus $T$, and Weyl group $W$. Let $X:=G/B$ be the full flag variety, and
$H^T_{*}(X)$, $H^{*}_T(X)$ denote the $T$-equivariant homology
and $T$-equivariant cohomology, respectively.
For any $w\in W$, there are
Schubert cells $X(w)^\circ=BwB/B$, $Y(w)^\circ=B^{-} w B/B$,
and their closures $X(w)=\overline{X(w)^\circ}$, $Y(w)=\overline{Y(w)^\circ}$ inside $X$.
We refer the readers to \cite{MNS20} and  \cite{MNS22} for unexplained terminology below.

There is a left Weyl group action on $H_*^T(G/B)$. For any $w\in W$, let us denote its action by $\delta_w$. Then we have the following correspondence:
\begin{tabular}{lll}
$\partial_w$ & Schubert class $[X(w)]\in H^T_{*}(X)$,\\[0.2cm]
$\cLL_w$ & \begin{minipage}{6cm}
CSM class of the Schubert cell $c_{SM}(X(w)^\circ)\in H^T_{*}(X),$
\end{minipage}\\[0.5cm]
$\eta^v$ & \begin{minipage}{6cm}
Segre–MacPherson class of the opposite Schubert cell  $s_{M}(Y(v)^\circ)\in H^{*}_T(X)_{loc}$,
\end{minipage}\\[0.5cm]
$\eta^v(w)$& the localization $s_{M}(Y(v)^\circ)|_w$.
\end{tabular}

To be more precise, the operator $\cLL_w$ becomes the left Demazure--Lusztig operator, denoted by
$\mathcal T_w^L$ in \cite{MNS20}, and it is shown in Theorem 4.4 of \textit{loc. cit.} that
\begin{equation}
\mathcal T_w^L (c_{SM}(X(id)^\circ))=c_{SM}(X(w)^\circ).
\end{equation}
The other identifications can be proved similarly.

{\bf Acknowledgements.} {The authors thank the referee for valuable comments.
L.~C.~Mihalcea was partially supported by NSF grant DMS-2152294 
and a Simons Collaboration grant; H.~Naruse was supported in part by JSPS KAKENHI Grant Number 16H03921.
}

\bibliographystyle{halpha}

\end{document}